\pdfoutput=1
\documentclass[12pt,letter]{amsart}
\usepackage{amsmath,amssymb,amsthm,amscd,amsxtra,amsfonts,mathrsfs,graphicx,enumerate,bm,slashed,array,mathtools,setspace,amsfonts}
\usepackage{xcolor}
\usepackage[all,color]{xy}
\input xy 
\xyoption{all}
\newtheorem{Theorem}{Theorem}[section]
\newtheorem{Lemma}[Theorem]{Lemma} 
\newtheorem{Proposition}[Theorem]{Proposition}
\newtheorem{Corollary}[Theorem]{Corollary}

\newtheorem{Non-example}[Theorem]{Non-example}
\newtheorem{Remark}[Theorem]{Remark}

\newtheorem{Definition}[Theorem]{Definition}

\newtheorem{Notation}[Theorem]{Notation}
\newtheorem*{Theorem A}{Theorem A}
\newtheorem*{Corollary B}{Corollary B}
\newtheorem*{Proposition C}{Proposition C}
\newtheorem*{Theorem D}{Theorem D}
\newtheorem*{Corollary E}{Corollary E}
\newtheorem*{Theorem F}{Theorem F}
\newtheorem*{Proposition G}{Proposition G}
\newtheorem*{Theorem H}{Theorem H}

\newcommand*{\overbar}[1]{\mkern 1.5mu\overline{\mkern-1.5mu#1\mkern-1.5mu}\mkern 1.5mu}
\advance\evensidemargin-.5in
\advance\oddsidemargin-.5in
\advance\textwidth1in

\begin{document}
 \author{Charlie Beil}
\address{Heilbronn Institute for Mathematical Research, School of Mathematics, Howard House, The University of Bristol, Bristol, BS8 1SN, United Kingdom.}
 \email{charlie.beil@bristol.ac.uk}
 \title{Morita equivalences and Azumaya loci from Higgsing dimer algebras}
 \date{}

\begin{abstract}
Let $\psi: A \to A'$ be a cyclic contraction of dimer algebras, with $A$ non-cancellative and $A'$ cancellative. 
$A'$ is then prime, noetherian, and a finitely generated module over its center.
In contrast, $A$ is often not prime, nonnoetherian, and an infinitely generated module over its center.
We present certain Morita equivalences that relate the representation theory of $A$ with that of $A'$.

We then characterize the Azumaya locus of $A$ in terms of the Azumaya locus of $A'$, and give an explicit classification of the simple $A$-modules parameterized by the Azumaya locus.
Furthermore, we show that if the smooth and Azumaya loci of $A'$ coincide, then the smooth and Azumaya loci of $A$ coincide.
This provides the first known class of algebras that are nonnoetherian and infinitely generated modules over their centers, with the property that their smooth and Azumaya loci coincide.
\end{abstract}

\maketitle

\tableofcontents

\section{Introduction}

We begin by recalling the definition of a dimer algebra, which is a type of quiver algebra whose quiver is dual to a dimer model.

\begin{Definition} \label{dimer def} \rm{ \
\begin{itemize}
 \item Let $Q$ be a finite quiver whose underlying graph $\overbar{Q}$ embeds into a two-dimensional real torus $T^2$, such that each connected component of $T^2 \setminus \overbar{Q}$ is simply connected and bounded by an oriented cycle of length at least 2, called a \textit{unit cycle}.\footnote{In contexts such as cluster algebras, it is useful to consider dimer algebras where $\overbar{Q}$ embeds into any compact surface; see for example \cite{BKM}.} 
The \textit{dimer algebra} of $Q$ is the quiver algebra $A = kQ/I$ with relations
$$I := \left\langle p - q \ | \ \exists \ a \in Q_1 \text{ such that } pa \text{ and } qa \text{ are unit cycles} \right\rangle \subset kQ,$$
where $p$ and $q$ are paths.
 \item Two paths $p,q \in A$ form a \textit{non-cancellative pair} if $p \not = q$, and there is a path $r \in A$ such that 
$$rp = rq \not = 0 \ \ \text{ or } \ \ pr = qr \not = 0.$$
$A$ is \textit{non-cancellative} if it contains a non-cancellative pair, and otherwise $A$ is \textit{cancellative}. 
\end{itemize}
}\end{Definition}

Cancellative dimer algebras (which are also called `consistent'\footnote{More precisely, let $A$ be a dimer algebra such that each arrow is contained in a perfect matching.  Then $A$ is cancellative if and only if $A$ satisfies a combinatorial `consistency condition' \cite[Theorem 1.1]{IU}, \cite[Theorem 6.2]{Bo}.}) are 3-Calabi-Yau algebras and noncommutative crepant resolutions of their 3-dimensional Gorenstein centers (e.g., \cite[Theorem 10.2]{Bo}, \cite{Br}, \cite[Theorem 4.3]{D}, \cite[Theorem 6.3]{MR}).
However, almost all dimer algebras are non-cancellative, and so it is of great interest to understand them. 

To study non-cancellative dimer algebras, the notion of a `cyclic contraction' was introduced in \cite[Section 4.1]{B2}.
This notion remains our primary tool in this paper.
We first define a contraction, which formalizes Higgsing in abelian quiver gauge theories. 

\begin{Definition} \label{contraction} \rm{ 
Let $Q = \left( Q_0, Q_1, \operatorname{t},\operatorname{h} \right)$ be a dimer quiver, with tail and head maps $\operatorname{t}, \operatorname{h}: Q_1 \to Q_0$.
Let $Q_1^* \subset Q_1$ be a subset of arrows. 
Consider the quiver $Q' = \left( Q'_0, Q'_1, \operatorname{t}',\operatorname{h}' \right)$ obtained by contracting each arrow in $Q_1^*$ to a vertex. 
Specifically,
$$Q_0' := Q_0 / \left\{ \operatorname{h}(a) \sim \operatorname{t}(a) \ | \ a \in Q_1^* \right\}, \ \ \ Q_1' = Q_1 \setminus Q_1^*,$$
and for each arrow $a \in Q'_1$, 
$$\operatorname{h}'(a) = \operatorname{h}(a) \ \ \ \text{ and } \ \ \ \operatorname{t}'(a) = \operatorname{t}(a).$$ 
There is a $k$-linear map of path algebras
$$\psi: kQ \rightarrow kQ'$$ 
defined by
$$\psi(a) = \left\{ \begin{array}{cl} a & \text{ if } \ a \in Q_0 \cup Q_1 \setminus Q_1^* \\ e_{\operatorname{t}(a)} & \text{ if } \ a \ \in Q_1^* \end{array} \right.$$
and extended multiplicatively to nonzero paths and $k$-linearly to $kQ$.
We call $\psi$ a \textit{contraction of dimer algebras} if $Q'$ is a dimer quiver, and $\psi$ induces a $k$-linear map of dimer algebras
$$\psi: A = kQ/I \to A' = kQ'/I'.$$
}\end{Definition}

We now describe the structure we wish to be preserved under a contraction. 

\begin{Definition} \label{simple matching} \rm{
Let $A = kQ/I$ be a dimer algebra.
\begin{itemize}
 \item A \textit{perfect matching} $D \subset Q_1$ is a set of arrows such that each unit cycle contains precisely one arrow in $D$.
 \item A \textit{simple matching} $D \subset Q_1$ is a perfect matching such that $Q \setminus D$ supports a simple $A$-module of dimension $1^{Q_0}$ (that is, $Q \setminus D$ contains a cycle that passes through each vertex of $Q$). 
 Denote by $\mathcal{S}$ the set of simple matchings of $A$. 
\end{itemize}
}\end{Definition}

Consider a contraction of dimer algebras $\psi: A \to A'$, where $A$ is non-cancellative and $A'$ is cancellative. 
Denote by 
$$B := k\left[ x_D \ | \ D \in \mathcal{S}' \right]$$
the polynomial ring generated by the simple matchings of $A'$.
Denote by $E_{ij}$ a square matrix with a $1$ in the $ij$-th slot and zeros elsewhere.
Since $A'$ is cancellative, there is an algebra monomorphism 
\begin{equation} \label{tau}
\tau: A' \to M_{|Q'_0|}(B)
\end{equation} 
defined on $i \in Q'_0$ and $a \in Q'_1$ by
$$\tau(e_i) := E_{ii}, \ \ \ \tau(a) :=  E_{\operatorname{h}(a),\operatorname{t}(a)} \prod_{a \in D \in \mathcal{S}'} x_D,$$
and extended multiplicatively and $k$-linearly to $A'$ \cite[Theorem 3.4]{B}.
For $p \in e_jA'e_i$, denote by $\bar{\tau}(p) \in B$ the single nonzero matrix entry of $\tau(p)$, that is,
$$\tau(p) = \bar{\tau}(p)E_{ji}.$$

\begin{Definition} \cite[Definition 4.3]{B} \rm{
Let $\psi: A \to A'$ be a contraction to a cancellative dimer algebra.
If 
$$S := k\left[ \cup_{i \in Q_0} \bar{\tau}\psi\left(e_iAe_i\right) \right] = k\left[ \cup_{i \in Q'_0} \bar{\tau}\left( e_iA'e_i \right) \right],$$
then we say $\psi$ is \textit{cyclic}, and call $S$ the \textit{cycle algebra} of $A$.
}\end{Definition}

Throughout, we will suppose $\psi: A \to A'$ is cyclic.
Consequently, the centers of $A$ and $A'$,
$$Z := Z(A) \ \ \ \text{ and } \ \ \ Z' := Z(A'),$$
both have Krull dimension 3 \cite[Theorem 4.66]{B2}. 
Furthermore, $Z'$ is noetherian and isomorphic to the cycle algebra $S$,
\begin{equation} \label{Z' S}
\begin{array}{rcl} 
Z' & \cong & S\\
z & \mapsto & \bar{\tau}(ze_i)
\end{array}
\end{equation}
where $i \in Q'_0$ is any vertex \cite[Theorems 3.3]{B2}. 
In contrast, $Z$ is nonnoetherian \cite[Theorem 4.45]{B2}.
Moreover, there is an embedding of algebras
\begin{equation} \label{mono}
\begin{array}{rcl}
\hat{Z} := Z/\operatorname{nil}Z & \hookrightarrow & R := k\left[ \cap_{i \in Q_0} \bar{\tau}\psi\left(e_iAe_i\right) \right] \subset S \subset B\\
z & \mapsto & \bar{\tau}\psi(ze_i)
\end{array}
\end{equation}
where again $i \in Q_0$ is any vertex \cite[Theorem 4.27]{B2}.
We will identify the reduced center $\hat{Z}$ of $A$ with its image in $R$.

\begin{Notation} \rm{
For $\mathfrak{p} \in \operatorname{Spec}Z$ (which necessarily contains the nilradical $\operatorname{nil}Z$), set 
\begin{equation} \label{prime notation}
\hat{\mathfrak{p}} := \mathfrak{p} + \operatorname{nil}Z \in \operatorname{Spec}\hat{Z}.
\end{equation} 
Conversely, for $\hat{\mathfrak{p}} \in \operatorname{Spec}\hat{Z}$, denote by $\mathfrak{p} \in \operatorname{Spec}Z$ the prime ideal satisfying (\ref{prime notation}).
}\end{Notation}

Using the notions of nonnoetherian geometry introduced in \cite{B1}, we may view the geometry of $Z$ as the affine variety $\operatorname{Max}S$, with precisely one `smeared-out' positive dimensional (closed) point $\mathfrak{z}_0 \in \operatorname{Max}Z$ \cite[Theorem 4.68.2]{B2}.
Specifically, $\mathfrak{z}_0$ is the common $Z$-annihilator of the vertex simple $A$-modules.
Furthermore, $Z$ is locally noetherian on the complement of $\mathfrak{z}_0$:
\begin{equation} \label{U z0}
\begin{array}{rcl}
U & := & \left\{ \mathfrak{n} \in \operatorname{Max}S \ | \ \hat{Z}_{\mathfrak{n} \cap \hat{Z}} = S_{\mathfrak{n}} \right\} \\
& \stackrel{\textsc{(i)}}{=} & \left\{ \mathfrak{n} \in \operatorname{Max}S \ | \ \hat{Z}_{\mathfrak{n} \cap \hat{Z}} \text{ is noetherian} \right\} \\
& \stackrel{\textsc{(ii)}}{=} & \left\{ \mathfrak{n} \in \operatorname{Max}S \ | \ \mathfrak{n} \cap \hat{Z} \not = \hat{\mathfrak{z}}_0 \right\},
\end{array}
\end{equation}
where (\textsc{i}) and (\textsc{ii}) hold by \cite[Theorem 4.63]{B2}. 
The locus $U$ captures the points where $\hat{Z}$ and $S$ look locally the same.

Recall that two rings are Morita equivalent if they have equivalent module categories. 
Our main results are the following.

\begin{Theorem}
(Theorems \ref{main1}, \ref{Azumaya 1}, \ref{Azumaya}, and Corollary \ref{main corollary}.) 
Let $\psi: A \to A'$ be a cyclic contraction.
\begin{enumerate}
 \item Let $\mathfrak{q} \in \operatorname{Spec}S$, and set $\hat{\mathfrak{p}} := \mathfrak{q} \cap \hat{Z}$.
Then the following are equivalent:
\begin{itemize}
 \item $\mathcal{Z}(\mathfrak{q}) \cap U \not = \emptyset$.
 \item The localizations 
$$A_{\mathfrak{p}} := A \otimes_Z Z_{\mathfrak{p}} \ \ \text{ and } \ \ A'_{\mathfrak{q}} := A' \otimes_{Z'} Z'_{\mathfrak{\mathfrak{q}}}$$
are Morita equivalent. 
 \item The localized algebra $A_{\mathfrak{p}}$ is prime, noetherian, and a finitely generated module over its center $Z_{\mathfrak{p}}$ with PI degree $|Q_0|$.
\end{itemize}
 \item The (noncommutative) function fields are Morita equivalent,
$$A \otimes_Z \operatorname{Frac}Z \ \sim \ A' \otimes_{Z'} \operatorname{Frac}Z' \ \sim \ \operatorname{Frac}Z \ \sim \ \operatorname{Frac}Z'.$$  
 \item The Azumaya locus $\mathcal{A} \subset \operatorname{Max}Z$ of $A$ coincides with the intersection of the Azumaya locus $\mathcal{A}' \subset \operatorname{Max}Z'$ of $A'$ and the locus $U \subset \operatorname{Max}Z'$,
$$\mathcal{A} \cong \mathcal{A}' \cap U.$$
\end{enumerate}
\end{Theorem}

From the third statement we obtain the first known class of algebras that are nonnoetherian and infinitely generated modules over their centers, with the property that their Azumaya and smooth loci coincide (Corollary \ref{Azumaya corollary}).
 
In addition, we give an explicit classification of the simple $A$-module isoclasses of dimension $1^{Q_0}$, or equivalently, the $A$-modules which sit over the Azumaya locus $\mathcal{A}$ of $A$ (Proposition \ref{simple 1} and Theorem \ref{almost impression}).
Finally, we show that the cycle algebra $S$ is unique (Theorem \ref{cycle algebra theorem}).
Specifically, we show that the cycle algebra is isomorphic to the $\operatorname{GL}$-invariant rings
$$S \cong k[ \overbar{\mathbb{S}(A)} ]^{\operatorname{GL}} = k[ \overbar{\mathbb{S}(A')} ]^{\operatorname{GL}},$$
where 
$$\mathbb{S}(A) \subset \operatorname{Rep}_{1^{Q_0}}(A) \ \ \ \text{ and } \ \ \ \mathbb{S}(A') \subset \operatorname{Rep}_{1^{Q'_0}}(A')$$ 
are the open subvarieties consisting of simple modules, and $\overbar{\mathbb{S}(A)}$ and $\overbar{\mathbb{S}(A')}$ are their Zariski closures.\\
\\
\textit{Conventions:}
Throughout, $k$ is an uncountable algebraically closed field of characteristic zero.
We will denote by $\operatorname{Frac}R$ the ring of fractions of $R$; by $\operatorname{Max}R$ the set of maximal ideals of $R$; by $\operatorname{Spec}R$ either the set of prime ideals of $R$ or the affine $k$-scheme with global sections $R$; by $R_{\mathfrak{p}}$ the localization of $R$ at $\mathfrak{p} \in \operatorname{Spec}R$; by $\operatorname{nil}R$ the nilradical of $R$; and by $\mathcal{Z}(\mathfrak{a})$ the closed set $\left\{ \mathfrak{m} \in \operatorname{Max}R \ | \ \mathfrak{m} \supseteq \mathfrak{a} \right\}$ of $\operatorname{Max}R$ defined by the subset $\mathfrak{a} \subset R$. 

We will denote by $Q = \left( Q_0,Q_1,\operatorname{t}, \operatorname{h} \right)$ a quiver with vertex set $Q_0$, arrow set $Q_1$, and head and tail maps $\operatorname{h},\operatorname{t}: Q_1 \to Q_0$.
We will denote by $kQ$ the path algebra of $Q$, and by $e_i$ the idempotent corresponding to vertex $i \in Q_0$.
Multiplication of paths is read right to left, following the composition of maps. 
A loop in a quiver is an arrow which is a cycle. 
By module we mean left module.
If $\rho: A \to \operatorname{End}_k(V)$ is a representation of $A$, then we will denote by $V_{\rho} := V$ the left $A$-module defined by $\rho$.
By infinitely generated $R$-module, we mean an $R$-module that is not finitely generated.
A ring is noetherian if it is both left and right noetherian. 
We will often write $rs$ for $r \otimes s \in R \otimes S$ if the tensor product is clear.
Finally, by non-constant monomial, we mean a monomial that is not in $k$.

\section{Morita equivalences} \label{Section 2}

Throughout, $A$ is a non-cancellative dimer algebra; $\psi: A \to A'$ is a cyclic contraction; and unless stated otherwise, $\mathfrak{q} \in \operatorname{Spec}S$ satisfies 
$$\mathcal{Z}(\mathfrak{q}) \cap U \not = \emptyset.$$
Set $\hat{\mathfrak{p}} := \mathfrak{q} \cap \hat{Z} \in \operatorname{Spec}\hat{Z}$. 

\begin{Notation} \rm{
Let $\pi: \mathbb{R}^2 \rightarrow T^2$ be a covering map such that for some $i \in Q_0$, 
$$\pi\left(\mathbb{Z}^2 \right) = i \in Q_0.$$
Denote by $Q^+ := \pi^{-1}(Q) \subset \mathbb{R}^2$ the covering quiver of $Q$.  
For each path $p$ in $Q$, denote by $p^+$ the unique path in $Q^+$ with tail in the unit square $[0,1) \times [0,1) \subset \mathbb{R}^2$ satisfying $\pi(p^+) = p$.

We denote by $\sigma_i \in A$ the unique unit cycle at vertex $i \in Q_0$.
By a \textit{cyclic subpath} of a path $p$, we mean a proper subpath of $p$ that is a non-vertex cycle.

We consider the following sets of cycles in $A$.
\begin{itemize}
 \item Let $\mathcal{C}$ be the set of cycles in $A$ (i.e., cycles in $Q$ modulo $I$).
 \item For $u \in \mathbb{Z}^2$, let $\mathcal{C}^u$ be the set of cycles $p \in \mathcal{C}$ such that 
$$\operatorname{h}(p^+) = \operatorname{t}(p^+) + u \in Q_0^+.$$
 \item For $i \in Q_0$, let $\mathcal{C}_i$ be the set of cycles in the vertex corner ring $e_iAe_i$.
 \item Let $\hat{\mathcal{C}}$ be the set of cycles $p \in \mathcal{C}$ such that $(p^2)^+$ does not have a cyclic subpath; or equivalently, the lift of each cyclic permutation of $p$ does not have a cyclic subpath.
\end{itemize}
We denote the intersection $\hat{\mathcal{C}} \cap \mathcal{C}^u \cap \mathcal{C}_i$, for example, by $\hat{\mathcal{C}}^u_i$.
}\end{Notation}

\begin{Notation} \rm{
For $p \in e_jAe_i$ and $q \in e_{\ell}A'e_k$, set 
$$\overbar{p} := \bar{\tau}\psi(p) \in B \ \ \ \text{ and } \ \ \ \overbar{q} := \bar{\tau}(q) \in B.$$
}\end{Notation}

\begin{Notation} \rm{
Set 
$$\begin{array}{rcl}
Q_1^{\mathcal{S}} & := & \left\{ a \in Q_1 \ | \ a \not \in D \text{ for each } D \in \mathcal{S} \right\}\\
& \stackrel{\textsc{(i)}}{=} & \left\{ a \in Q_1 \ | \ \rho(a) \not = 0 \text{ for each simple representation } \rho \text{ of dimension } 1^{Q_0} \right\},
\end{array}$$
where (\textsc{i}) holds by \cite[Lemma 4.39]{B2}.
}\end{Notation}

\begin{Lemma} \label{sr lemma}
Let $z$ be a non-constant monomial in $Z$ which is not divisible by $\sigma$.
Let $r$ be a path whose arrow subpaths are all in $Q_1^{\mathcal{S}}$.
Then there is a path $s$ such that 
$$ze_{\operatorname{t}(r)} = sr.$$
\end{Lemma}

\begin{proof}
(i) First suppose $r$ is an arrow, $r = \delta \in Q_1^{\mathcal{S}}$.
Let $z \in Z$ be a non-constant monomial such that $\sigma \nmid z$.
Since $\hat{Z} \subseteq R$ by (\ref{mono}), there are cycles
$$p \in \mathcal{C}^u_{\operatorname{h}(\delta)} \ \ \ \text{ and } \ \ \ q \in \mathcal{C}^v_{\operatorname{t}(\delta)}$$
such that $\overbar{p} = \overbar{q} = z$.
Since $\sigma \nmid g = \overbar{p} = \overbar{q}$, $p$ and $q$ are in $\hat{\mathcal{C}}$ by \cite[Lemma 4.11.2]{B2}.
In particular, $u$ and $v$ are nonzero.
Whence $u = v$ by \cite[Lemma 4.13]{B2}.
Therefore $p$ and $q$ are in $\hat{\mathcal{C}}^u$.
Thus the paths $(p \delta)^+$ and $(\delta q)^+$ bound a compact region 
$$\mathcal{R}_{p \delta,\delta q} \subset \mathbb{R}^2.$$
Furthermore, since $z \in R$, we may suppose that the interior of $\mathcal{R}_{\delta p,q\delta}$ does not contain any vertices of $Q^+$, by \cite[Lemma 2.15]{B2}. 

First suppose $p^+$ and $q^+$ do not intersect. 
Then $\delta$ is contained in a simple matching $D$ of $A$ by \cite[Lemma 2.16]{B2}; see Figure \ref{r's figure}.i.
But this is a contradiction since $\delta \in Q_1^{\mathcal{S}}$.

Therefore $p^+$ and $q^+$ intersect at a vertex $i^+$; see Figure \ref{r's figure}.ii. 
By assumption, $\sigma \nmid z$.
Thus, since $\overbar{\delta} = 1$,
$$\overbar{p}_1 = \overbar{q}_1 \ \ \text{ and } \ \ \overbar{p}_2 = \overbar{q}_2,$$
by \cite[Lemma 2.3.2]{B2}.
Whence
$$\overbar{q_2 p_1 \delta} = \overbar{q}_2 \overbar{q}_1 = \overbar{q} = z.$$
Therefore, since $z \in R$ and $\sigma \nmid z$,
$$q_2p_1\delta \in Ze_{\operatorname{t}(\delta)},$$ 
by \cite[Proposition 4.30.1]{B2}. 
In particular, we may take $s = q_2p_1$.

(ii) Now suppose $r = \delta_{\ell} \cdots \delta_2\delta_1 \not = 0$, with each $\delta_i \in Q_1^{\mathcal{S}}$.
By Claim (1), for each $1 \leq i \leq \ell$ there is a central element $z_i \in Z$ such that 
$$z_ie_{\operatorname{t}(\delta_i)} = s_i\delta_i.$$
Set $s := s_1s_2 \cdots s_{\ell}$.
Then the central element $z := z_{\ell} \cdots z_2z_1$ satisfies
$$\begin{array}{rcl}
ze_{\operatorname{t}(s)} & = & z_{\ell} \cdots z_3 z_2 (s_1\delta_1) \\
& = & z_{\ell} \cdots z_3s_1 z_2 \delta_1 \\
& = & z_{\ell} \cdots z_3s_1 s_2 \delta_2\delta_1 \\
& \vdots & \\
& = & s_1 s_2 \cdots s_{\ell}\delta_{\ell} \cdots \delta_2 \delta_1 \\
& = & sr.
\end{array}$$
\end{proof}

\begin{figure}
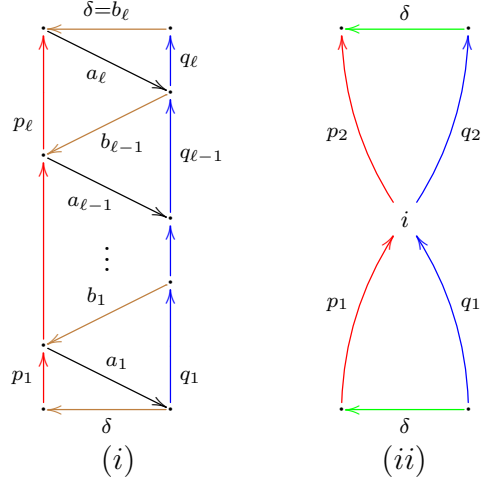

$$\begin{array}{ccc}
\xy 0;/r.4pc/:
(-5,-15)*{\cdot}="1";(5,-15)*{\cdot}="1'";
(-5,-10)*{\cdot}="2";(5,-5)*{\cdot}="3";
(0,-2.5)*{\vdots}="";
(5,0)*{\cdot}="6";(-5,5)*{\cdot}="5";
(5,10)*{\cdot}="7";
(-5,15)*{\cdot}="8";(5,15)*{\cdot}="8'";
(5,-10)*{}="9";(-5,10)*{}="10";
{\ar^{\delta}@[brown]"1'";"1"};{\ar_{\delta = b_{\ell}}@[brown]"8'";"8"};
{\ar^{p_1}@[red]"1";"2"};{\ar^{a_1}"2";"1'"};{\ar@{-}_{q_1}@[blue]"1'";"9"};
{\ar^{}@[blue]"9";"3"};{\ar_{b_1}@[brown]"3";"2"};{\ar^{}@[red]"2";"5"};{\ar^{}@[blue]"3";"6"};
{\ar_{q_{\ell-1}}@[blue]"6";"7"};{\ar^{b_{\ell-1}}@[brown]"7";"5"};{\ar@{-}^{p_{\ell}}@[red]"5";"10"};{\ar^{}@[red]"10";"8"};
{\ar_{a_{\ell}}"8";"7"};{\ar_{q_{\ell}}@[blue]"7";"8'"};{\ar_{a_{\ell-1}}"5";"6"};
\endxy
& \ \ \ \ &
\xy 0;/r.4pc/:
(-5,-15)*{\cdot}="1";(5,-15)*{\cdot}="1'";
(-5,15)*{\cdot}="8";(5,15)*{\cdot}="8'";
(0,0)*+{\text{\scriptsize{$i$}}}="2";
{\ar@/^/^{p_1}@[red]"1";"2"};
{\ar@/^/^{p_2}@[red]"2";"8"};
{\ar@/_/_{q_1}@[blue]"1'";"2"};
{\ar@/_/_{q_2}@[blue]"2";"8'"};
{\ar^{\delta}@[green]"1'";"1"};{\ar_{\delta}@[green]"8'";"8"};
\endxy
\\
(i) & & (ii)
\end{array}$$
\caption{Cases for Lemma \ref{sr lemma}.  
In case (i), $p$ and $q$ factor into paths $p = p_{\ell} \cdots p_2p_1$ and $q = q_{\ell} \cdots q_2q_1$, where $a_1, \ldots, a_{\ell}, b_1, \ldots, b_{\ell}$ are arrows, and the cycles $a_jb_jq_j$ and $b_{j-1}a_jp_j$ are unit cycles.
By \cite[Lemma 2.16]{B2}, the $b_j$ arrows, drawn in brown, belong to a simple matching $D$ of $A$.
In case (ii), $p$ and $q$ factor into paths $p = p_2e_ip_1$ and $q = q_2 e_iq_1$.}
\label{r's figure}
\end{figure}

Denote the origin of $\operatorname{Max}S$ by
$$\mathfrak{n}_0  := \left( \overbar{s} \in S \ | \ s \text{ a non-vertex cycle} \right)S \in \operatorname{Max}S.$$
Then the $Z$-annihilator $\mathfrak{z}_0 \in \operatorname{Max}Z$ of the vertex simple $A$-modules satisfies
$$\hat{\mathfrak{z}}_0 = \mathfrak{n}_0 \cap \hat{Z} \in \operatorname{Max}\hat{Z}.$$

\begin{Lemma} \label{invertible}
Suppose each non-constant monomial in $\hat{Z}$ is divisible by $\sigma$.
Then each monomial in $\hat{Z}$ is invertible in the localization $\hat{Z}_{\mathfrak{q} \cap \hat{Z}}$.
\end{Lemma}

\begin{proof}
Suppose the hypotheses hold.
Then $\mathfrak{q}$ is contained in some maximal ideal $\mathfrak{n} \in \operatorname{Max}S \setminus \{ \mathfrak{n}_0 \}$.

Assume to the contrary that $\hat{\mathfrak{m}} := \mathfrak{n} \cap \hat{Z}$ contains a monomial.
Then $\hat{\mathfrak{m}} = \hat{\mathfrak{z}}_0$ by \cite[Lemma 4.56]{B2}.
Whence 
$$\mathfrak{n} \cap \hat{Z} = \hat{\mathfrak{m}} = \hat{\mathfrak{z}}_0 = \mathfrak{n}_0 \cap \hat{Z}.$$
But then $\mathfrak{n} = \mathfrak{n}_0$ by \cite[Theorem 2.5.1]{B1}, a contradiction.
Therefore each monomial in $\hat{Z}$ is not in $\hat{\mathfrak{m}}$.
In particular, each monomial in $\hat{Z}$ is not in $\mathfrak{q} \cap \hat{Z} \subseteq \hat{\mathfrak{m}}$.
\end{proof}

\begin{Theorem} \label{=0}
Suppose $\mathcal{Z}(\mathfrak{q}) \cap U \not = \emptyset$.
If $p,q \in A$ is a non-cancellative pair, then $p = q$ in the localization $A_{\mathfrak{p}}$.
In particular,
$$\operatorname{nil}Z \cdot Z_{\mathfrak{p}} = 0.$$
\end{Theorem}

\begin{proof}
Fix $\mathfrak{n} \in \mathcal{Z}(\mathfrak{q}) \cap U$; then $\hat{\mathfrak{p}}$ is contained in $\hat{\mathfrak{m}} := \mathfrak{n} \cap \hat{Z} \in \operatorname{Max}\hat{Z}$.

(i) Consider a non-cancellative pair $p,q \in e_jAe_i$.
We claim that $p = q$ in $A_{\mathfrak{p}}$. 
It suffices to suppose that there is no non-cancellative pair $s^+,t^+$ which bounds a region $\mathcal{R}_{s,t} \subset \mathbb{R}^2$ properly contained in the region $\mathcal{R}_{p,q}$ bounded by $p^+,q^+$.\footnote{Such a pair $p,q$ is called \textit{minimal} in \cite{B2}.} 
Let $r$ be a path of minimal length such that $rp = rq \not = 0$.
Then by \cite[Proposition 4.37]{B2}, each arrow subpath of $r$ is in $Q_1^{\mathcal{S}}$.

(i.a) First suppose there is a non-constant monomial in $\hat{Z}$ which is not divisible (in $B$) by $\sigma$.
Then by \cite[Lemma 4.58]{B2}, there is a non-constant monomial $z \in Z \setminus \mathfrak{m}$ such that $\sigma \nmid z$. 
Thus by Lemma \ref{sr lemma}, there is a path $s$ such that 
$$ze_{\operatorname{t}(r)} = sr.$$ 
Therefore in $A_{\mathfrak{p}}$,
$$p - q = \frac{z}{z}(p-q) = \frac{sr}{z} (p-q) = \frac{s}{z} r(p-q) = 0.$$

(i.b) Now suppose every non-constant monomial in $\hat{Z}$ is divisible by $\sigma$.
Let $s$ be path from $\operatorname{h}(r)$ to $\operatorname{t}(r)$ that passes through each vertex of $Q$.
Then $\overbar{sr} \in R$.
Thus there is some $n \geq 1$ such that 
$$(sr)^n \in Ze_{\operatorname{t}(r)},$$ 
by \cite[Proposition 4.30.3]{B2}.
Let $z \in Z$ be such that $ze_{\operatorname{t}(r)} = (sr)^n$.
Then $z$ is invertible in $Z_{\mathfrak{p}}$ by Lemma \ref{invertible}.
Therefore in $A_{\mathfrak{p}}$,
$$p - q = \frac{z}{z}(p-q) = \frac{(sr)^{n-1}sr}{z} (p-q) = \frac{(sr)^{n-1}s}{z} r(p-q) = 0.$$

(ii) By \cite[Theorem 4.24]{B2},
$$\operatorname{nil}Z = Z \cap \operatorname{ker}\psi.$$ 
Therefore by Claim (i),
$$\operatorname{nil}Z \cdot Z_{\mathfrak{p}} = \operatorname{ker}\psi|_Z \cdot Z_{\mathfrak{p}} = 0.$$
\end{proof}

\begin{Remark} \rm{
The \textit{homotopy algebra} of a dimer algebra $A$ is the quotient
$$\tilde{A} := A/\left\langle p - q \ | \ p,q \text{ is a non-cancellative pair} \right\rangle,$$
introduced in \cite[Definition 4.33]{B2}.
In what follows, the results for $A$ that depend on Theorem \ref{=0} also hold for its homotopy algebra $\tilde{A}$, since $\tilde{A}$ is cancellative by construction.
}\end{Remark}

\begin{Lemma} \label{most}
There is an algebra isomorphism
$$A \otimes_Z \hat{Z}_{\hat{\mathfrak{p}}} \cong A \otimes_Z Z_{\mathfrak{p}}.$$
\end{Lemma}

\begin{proof}
(i) We first claim that
$$A_{\mathfrak{p}} \otimes_Z \hat{Z} \cong A \otimes_Z \hat{Z}_{\hat{\mathfrak{p}}}.$$
It suffices to show that 
$$Z_{\mathfrak{p}} \otimes_Z \hat{Z} \cong \hat{Z}_{\hat{\mathfrak{p}}}.$$
Let $z \in Z$ and set $\hat{z} := z + \operatorname{nil}Z$.
If $z^{-1} \in Z_{\mathfrak{p}}$, then $z \not \in \mathfrak{p}$.
But $z \not \in \operatorname{nil}Z$ since $\operatorname{nil}Z \subset \mathfrak{p}$.
Thus $\hat{z}^{-1} \in \hat{Z}_{\hat{\mathfrak{p}}}$.
Conversely, if $\hat{z}^{-1} \in \hat{Z}_{\hat{\mathfrak{p}}}$, then $z \not \in \operatorname{nil}Z \cup \mathfrak{p}$.
Whence $z \not \in \mathfrak{p}$.
Therefore $z^{-1} \in Z_{\mathfrak{p}}$.

(ii) Applying the right exact functor $A_{\mathfrak{p}} \otimes_Z -$ to the short exact sequence
$$0 \to \operatorname{nil}Z \to Z \to \hat{Z} \to 0$$
we obtain the exact sequence
$$A_{\mathfrak{p}} \otimes_Z \operatorname{nil}Z \to A_{\mathfrak{p}} \otimes_Z Z \cong A_{\mathfrak{p}} \to A_{\mathfrak{p}} \otimes_Z \hat{Z} \stackrel{\textsc{(i)}}{\cong} A \otimes_Z \hat{Z}_{\hat{\mathfrak{p}}} \to 0,$$
where (\textsc{i}) holds by Claim (i).
But the left-most term is zero by Theorem \ref{=0}.
Therefore $A \otimes_Z \hat{Z}_{\hat{\mathfrak{p}}} \cong A_{\mathfrak{p}}$.
\end{proof}

\begin{Lemma} \label{central iso}
The restriction $\psi: Z \to A'$ induces an algebra isomorphism 
$$Z_{\mathfrak{p}} \cong \hat{Z}_{\hat{\mathfrak{p}}} \stackrel{\cong}{\longrightarrow} Z'_{\mathfrak{q}}.$$
\end{Lemma}

\begin{proof}
There are algebra isomorphisms
$$Z_{\mathfrak{p}} \stackrel{\textsc{(i)}}{\cong} \hat{Z}_{\hat{\mathfrak{p}}} \stackrel{\textsc{(ii)}}{\cong} S_{\mathfrak{q}} \stackrel{\textsc{(iii)}}{\cong} Z'_{\mathfrak{q}}.$$
Indeed, (\textsc{i}) holds by Theorem \ref{=0}.
(\textsc{ii}) holds since $S$ is a depiction of $\hat{Z}$ \cite[Theorem 4.68.1]{B2}, and our assumption that $\mathcal{Z}\left(\mathfrak{q}\right) \cap U \not = \emptyset$. 
Finally, (\textsc{iii}) holds by (\ref{Z' S}).
\end{proof}

It follows from Lemma \ref{central iso} that the contraction $\psi: A \to A'$ extends to a $k$-linear map
$$\psi_{\mathfrak{p}}: A_{\mathfrak{p}} = A \otimes_Z Z_{\mathfrak{p}} \longrightarrow A'_{\mathfrak{q}} = A' \otimes_{Z'} Z'_{\mathfrak{q}},$$
where the restriction of $\psi_{\mathfrak{p}}$ to $Z_{\mathfrak{p}}$ is the isomorphism $Z_{\mathfrak{p}} \to Z'_{\mathfrak{q}}$.
We call this extension a \textit{localized contraction}.

\begin{Remark} \label{*S} \rm{
If a contraction $\psi: A \to A'$ is cyclic, then 
$$Q^*_1 \subseteq Q_1^{\mathcal{S}}$$
by \cite[Theorem 4.38]{B2}.
In words, no contracted arrow is represented by zero in any simple representation of $A$ of dimension vector $1^{Q_0}$.
}\end{Remark}

\begin{Proposition} \label{delta*}
Suppose $\delta \in Q_1^*$, or more generally, $\delta \in Q_1^{\mathcal{S}}$.
Then $A_{\mathfrak{p}}$ contains an element $\delta^* \in e_{\operatorname{t}(\delta)} A_{\mathfrak{p}} e_{\operatorname{h}(\delta)}$ satisfying
\begin{equation} \label{delta* eq}
\delta^* \delta = e_{\operatorname{t}(\delta)} \ \ \text{ and } \ \ \delta \delta^* = e_{\operatorname{h}(\delta)}.
\end{equation}
Furthermore, if $\delta \in Q_1^*$, then
$$\psi_{\mathfrak{p}}(\delta^*) = \psi_{\mathfrak{p}}(\delta) \in Q'_0.$$
\end{Proposition}

\begin{proof}
Suppose $\delta \in Q_1^{\mathcal{S}}$.
Fix $\mathfrak{n} \in \mathcal{Z}(\mathfrak{q}) \cap U$; then $\hat{\mathfrak{p}}$ is contained in $\hat{\mathfrak{m}} := \mathfrak{n} \cap \hat{Z} \in \operatorname{Max}\hat{Z}$. 

By Claims (i.a) and (i.b) in the proof of Theorem \ref{=0} (with $r = \delta$), there is some $z \in Z \setminus \mathfrak{m}$ and a path $s$ such that $$ze_{\operatorname{t}(\delta)} = s \delta.$$ 
Set
$$\delta^* := \frac{s}{z} \in A_{\mathfrak{p}}.$$
Then in $A_{\mathfrak{p}}$,
$$\delta^* \delta = \frac{s}{z} \delta = \frac{ze_{\operatorname{t}(\delta)}}{z} = \frac zz e_{\operatorname{t}(\delta)} = e_{\operatorname{t}(\delta)}.$$
Similarly, $\delta \delta^* = e_{\operatorname{h}(\delta)}$.

Finally, set $w := \psi(z)$.
Then
$$\psi_{\mathfrak{p}}(\delta^*) = \psi_{\mathfrak{p}}(s) w^{-1} = \psi_{\mathfrak{p}}(s \delta) w^{-1}
= \psi_{\mathfrak{p}}(e_{\operatorname{t}(\delta)}z)w^{-1} = \psi_{\mathfrak{p}}(e_{\operatorname{t}(\delta)})ww^{-1} = \psi_{\mathfrak{p}}(\delta).$$
\end{proof}

\begin{Remark} \rm{
Cyclic contractions are not surjective in general.
Indeed, if a cyclic contraction $\psi$ contracts a unit cycle to a removable 2-cycle, then $\psi$ is not surjective; see \cite[Remark 4.2]{B2}.
}\end{Remark}

\begin{Lemma} \label{surjective}
The localized contraction $\psi_{\mathfrak{p}}: A_{\mathfrak{p}} \to A'_{\mathfrak{q}}$ is surjective, but not injective.
\end{Lemma}

\begin{proof}
(i) We first claim that 
\begin{equation} \label{A' subset}
A' \subseteq \psi_{\mathfrak{p}}(A_{\mathfrak{p}}).
\end{equation}
Indeed, let $q \in A'$ be a path which is not in the $\psi$-image of $A$.
Then $q$ factors into paths
$$q = q_n \psi(\delta_{n-1}) q_{n-1} \cdots q_2\psi(\delta_1)q_1,$$
where for each $i$, there is a path $p_i \in A$ such that $\psi(p_i) = q_i$; $\delta_i \in Q_1^*$; and 
$$\operatorname{h}(\delta_i) = \operatorname{h}(p_i) \ \ \text{ and } \ \ \operatorname{t}(\delta_i) = \operatorname{t}(p_{i+1}).$$ 
Furthermore, for each $i$ there is an `arrow' $\delta^*_i$, with opposite orientation to $\delta_i$, satisfying
$$\psi_{\mathfrak{p}}(\delta^*_i) = \psi_{\mathfrak{p}}(\delta_i) \in Q'_0,$$
by Proposition \ref{delta*}.
Thus $q$ is the $\psi_{\mathfrak{p}}$-image of the element
$$p = p_n \delta_{n-1}^* p_{n-1} \cdots p_2 \delta_1^* p_1 \in A_{\mathfrak{p}}.$$
Therefore (\ref{A' subset}) holds.

Furthermore, by the definition of $\psi_{\mathfrak{p}}$,
$$Z'_{\mathfrak{q}} = \psi_{\mathfrak{p}}(Z_{\mathfrak{p}}) \subseteq \psi_{\mathfrak{p}}(A_{\mathfrak{p}}).$$
Therefore $A'_{\mathfrak{q}} = A' \otimes_{Z'} Z'_{\mathfrak{q}} \subseteq \psi_{\mathfrak{p}}(A_{\mathfrak{p}})$. 

(ii) $\psi_{\mathfrak{p}}$ is not injective since $\psi_{\mathfrak{p}}(e_{\operatorname{t}(\delta)}) = \psi_{\mathfrak{p}}(e_{\operatorname{h}(\delta)})$.
\end{proof}

For the following, set
$$\epsilon_0 := 1_A - \sum_{\delta \in Q_1^*} e_{\operatorname{h}(\delta)}.$$ 
We note that a non-trivial contraction can never be an algebra homomorphism \cite[Remark 4.10]{B2}.

\begin{Proposition} \label{homomorphism on corner} \ 
\begin{enumerate}
 \item The map
$$\psi: \epsilon_0 A \epsilon_0 \to A'$$
is an algebra homomorphism.
Furthermore, its localization
\begin{equation} \label{localized restriction}
\psi_{\mathfrak{p}}: \epsilon_0 A_{\mathfrak{p}} \epsilon_0 \to A'_{\mathfrak{q}}
\end{equation}
is an algebra isomorphism.
 \item For each $i,j \in Q_0$, the $k$-linear map
\begin{equation} \label{restriction to corner}
\psi_{\mathfrak{p}}: e_i A_{\mathfrak{p}} e_j \to e_{\psi(i)}A'_{\mathfrak{q}}e_{\psi(j)}
\end{equation}
is bijective.
Consequently, if $i = j$, then it is an algebra isomorphism.
\end{enumerate}
\end{Proposition}

\begin{proof}
(1.i) We first claim that the restriction $\psi|_{\epsilon_0 A \epsilon_0}$ is an algebra homomorphism.
Since $\psi$ is a $k$-linear map, it suffices to show that the restriction is multiplicative on paths.

Let $p,q \in \epsilon_0 A \epsilon_0$ be paths. 
First suppose $\psi(q)\psi(p) \not = 0$.
Then 
$$\operatorname{h}(\psi(p)) = \operatorname{t}(\psi(q)) \in Q'_0.$$
Thus $\operatorname{h}(p) = \operatorname{t}(q)$ since $p,q \in \epsilon_0 A \epsilon_0$.
Whence $qp \not = 0$.
Therefore $\psi(qp) = \psi(q) \psi(p)$.

Now suppose $\psi(q) \psi(p) = 0$.
Then $\operatorname{h}(\psi(p)) \not = \operatorname{t}(\psi(q))$.
In particular, $\operatorname{h}(p) \not = \operatorname{t}(q)$.
Therefore $\psi(qp) = \psi(0) = 0 = \psi(q) \psi(p)$.

(1.ii) We now claim that the map (\ref{localized restriction}) is an algebra isomorphism. 
Indeed, (\ref{localized restriction}) is an algebra homomorphism by Claim (i).
Furthermore, 
$$\psi(\epsilon_0) = \sum_{i \in Q'_0} e_i = 1_{A'}.$$
Thus (\ref{localized restriction}) is surjective by Lemma \ref{surjective}.

To show injectivity, let $\delta \in Q_1^*$. 
The $\psi_{\mathfrak{p}}$-preimage of the vertex $\psi(\delta) \in Q_0'$ consists of the four elements
$$\delta, \ \delta^*, \ e_{\operatorname{t}(\delta)}, \ e_{\operatorname{h}(\delta)},$$ 
by Proposition \ref{delta*}.
If the tail of $\delta$ is not the head of another contracted arrow in $A$, then the only one of these elements in the corner ring $\epsilon_0 A_{\mathfrak{p}} \epsilon_0$ is the idempotent $e_{\operatorname{t}(\delta)}$, and otherwise none of them are in $\epsilon_0 A_{\mathfrak{p}} \epsilon_0$.
Therefore (\ref{localized restriction}) is injective by Theorem \ref{=0} and (\ref{delta* eq}).
This proves our claim.

(2) The map (\ref{restriction to corner}) is injective by Theorem \ref{=0} and (\ref{delta* eq}), and surjective by Lemma \ref{surjective}. 
Furthermore, the restriction of $\psi$ to the vertex corner ring $e_iAe_i$,
$$\psi: e_iAe_i \to A',$$ 
is an algebra homomorphism. 
\end{proof}

\begin{Theorem} \label{main1}
Let $\psi: A \to A'$ be a cyclic contraction of dimer algebras.
Then the localizations $A_{\mathfrak{p}}$ and $A'_{\mathfrak{q}}$ are Morita equivalent if and only if $\mathcal{Z}(\mathfrak{q}) \cap U \not = \emptyset$.
\end{Theorem}

\begin{proof}
Set
$$\mathsf{P} := A_{\mathfrak{p}} = A \otimes_Z Z_{\mathfrak{p}} \ \ \ \text{ and } \ \ \ \mathsf{Q} := A'_{\mathfrak{q}} = A' \otimes_{Z'} Z'_{\mathfrak{q}}.$$

First suppose $\mathcal{Z}(\mathfrak{q}) \cap U = \emptyset$.
Enumerate the contracted arrows,
$$Q_1^* = \left\{ \delta_1, \ldots, \delta_n \right\} \subset Q_1.$$
For $1 \leq i \leq n$, set
\begin{equation} \label{epsilon}
\epsilon_i := e_{\operatorname{h}(\delta_i)} \in Q_0 \ \ \ \text{ and } \ \ \ \epsilon'_i := \psi(\delta_i) \in Q'_0.
\end{equation}
Furthermore, set
\begin{equation} \label{epsilon2}
\epsilon_0 := 1_A - \sum_{i = 1}^n \epsilon_i \ \ \ \text{ and } \ \ \ \epsilon'_0 := 1_{A'}.
\end{equation}
Then $\mathsf{P}$ is isomorphic to the $n \times n$ tiled matrix algebra,
\begin{equation} \label{P iso}
\mathsf{P} \cong \left[ \epsilon_i \mathsf{P} \epsilon_j \right]_{ij}.
\end{equation}

Consider the $k$-linear map
\begin{equation} \label{P}
\zeta: \left[ \epsilon_i \mathsf{P} \epsilon_j \right]_{ij} \longrightarrow \left[ \epsilon'_i \mathsf{Q} \epsilon'_j \right]_{ij},
\end{equation}
defined by sending the $ij$-th entry $\alpha \in \epsilon_i \mathsf{P} \epsilon_j$ to the $ij$-th entry $\psi_{\mathfrak{p}}(\alpha) \in \epsilon'_i \mathsf{Q} \epsilon'_j$.
This map is an algebra isomorphism by Proposition \ref{homomorphism on corner}.
Furthermore, using the isomorphism (\ref{P iso}), we may view $\zeta$ as an algebra isomorphism 
$$\zeta: \mathsf{P} \stackrel{\cong}{\longrightarrow} \left[ \epsilon'_i \mathsf{Q} \epsilon'_j \right]_{ij}.$$

Now consider the bimodules
\begin{equation} \label{M}
_{\mathsf{P}}M_{\mathsf{Q}} = \left[ \begin{array}{c} \epsilon'_0 \mathsf{Q} \\ \epsilon'_1 \mathsf{Q} \\ \vdots  \\ \epsilon'_n \mathsf{Q} \end{array} \right] \ \ \text{ and } \ \ 
_{\mathsf{Q}}N_{\mathsf{P}} = \left[ \begin{array}{cccc} \mathsf{Q}\epsilon'_0 & \mathsf{Q} \epsilon'_1 & \cdots & \mathsf{Q}\epsilon'_n \end{array} \right],
\end{equation}
where $\mathsf{P}$ acts via the isomorphism $\mathsf{P} \cong \zeta(\mathsf{P})$.
The $\mathsf{P},\mathsf{P}$-bimodule homomorphism
$$\theta: M \otimes_{\mathsf{Q}} N \to \zeta(\mathsf{P}) \cong \mathsf{P}$$
defined by
$$\left[ \begin{array}{c} \epsilon'_0 s_0 \\ \epsilon'_1 s_1 \\ \vdots \\ \epsilon'_n s_n \end{array} \right] \otimes \left[ \begin{array}{cccc} t_0 \epsilon'_0 & t_1 \epsilon'_1 & \cdots & t_n\epsilon'_n \end{array} \right] \mapsto 
\left[ \epsilon'_i s_i t_j \epsilon'_j \right]_{ij}$$ 
is clearly surjective.
Furthermore, the $\mathsf{Q},\mathsf{Q}$-bimodule homomorphism
\begin{equation} \label{phi}
\phi: N \otimes_{\mathsf{P}} M \to \mathsf{Q}
\end{equation}
defined by
$$\left[ \begin{array}{cccc} t_0 \epsilon'_0 & t_1 \epsilon'_1 & \cdots & t_n\epsilon'_n \end{array} \right] \otimes \left[ \begin{array}{c} \epsilon'_0 s_0 \\ \epsilon'_1 s_1 \\ \vdots \\ \epsilon'_n s_n \end{array} \right] \mapsto \sum_{i = 0}^n t_i\epsilon'_i \epsilon'_i s_i = \sum_{i = 0}^n t_is_i$$
is also surjective.
Thus, since $\mathsf{P}$ and $\mathsf{Q}$ are unital, $\theta$ and $\phi$ are bimodule isomorphisms \cite[Lemma 4.5.2]{C},
$$\mathsf{P} \cong M \otimes_{\mathsf{Q}} N \ \ \ \text{ and } \ \ \ \mathsf{Q} \cong N \otimes_{\mathsf{P}} M.$$
Therefore $\mathsf{P}$ and $\mathsf{Q}$ are Morita equivalent, with progenerators $N$ and $M$.

Conversely, suppose $\mathcal{Z}(\mathfrak{q}) \cap U = \emptyset$. 
Then $\hat{Z}_{\hat{\mathfrak{p}}} \not = S_{\mathfrak{q}}$.
Furthermore, $S \cong Z'$ by (\ref{Z' S}).
Whence the centers of $\mathsf{P}$ and $\mathsf{Q}$ are not isomorphic:
$$Z(\mathsf{P}) = \hat{Z}_{\hat{\mathfrak{p}}} \not = S_{\mathfrak{q}} \cong Z(\mathsf{Q}).$$
But $\mathsf{P}$ and $\mathsf{Q}$ are unital.
Therefore $\mathsf{P}$ and $\mathsf{Q}$ are not Morita equivalent \cite[Theorem 5.9.iii]{McR}. 
\end{proof}

Although $Z$ may not be not be reduced, its reduction $\hat{Z} = Z/\operatorname{nil}Z$ is an integral domain \cite[Corollary 4.28]{B2}.

\begin{Corollary} \label{main corollary} \
\begin{enumerate}
 \item The (noncommutative) function fields 
$$A \otimes_Z \operatorname{Frac}\hat{Z}, \ \ \ A' \otimes_{Z'} \operatorname{Frac}Z', \ \ \ \operatorname{Frac}\hat{Z}, \ \ \ \operatorname{Frac}Z',$$
are Morita equivalent.
 \item If $\mathcal{Z}(\mathfrak{q}) \cap U \not = \emptyset$, then the noncommutative residue fields
$$A_{\mathfrak{p}}/\mathfrak{p} \ \ \ \text{ and } \ \ \ A'_{\mathfrak{q}}/\mathfrak{q}$$
are Morita equivalent.
\end{enumerate}
\end{Corollary}

\begin{proof}
(1) We have the following Morita equivalences:
$$A \otimes_Z \operatorname{Frac}\hat{Z} 
\ \stackrel{\textsc{(i)}}{\sim} \ A' \otimes_{Z'} \operatorname{Frac}Z' 
\ \stackrel{\textsc{(ii)}}{\sim} \ \operatorname{Frac}Z'
\ \stackrel{\textsc{(iii)}}{=} \ \operatorname{Frac}S
\ \stackrel{\textsc{(iv)}}{=} \ \operatorname{Frac}\hat{Z}.$$
Indeed, (\textsc{i}) follows from Theorem \ref{main1}.
(\textsc{ii}) holds since $A'$ is a cancellative, whence a noncommutative crepant resolution, and thus an endomorphism ring of a finitely generated projective $Z'$-module.
(\textsc{iii}) holds by (\ref{Z' S}).
Finally, (\textsc{iv}) holds since $S$ is a depiction of $\hat{Z}$ \cite[Theorem 4.68.1]{B2}.

(2) Recall the bimodule isomorphism $\phi$ defined in (\ref{phi}).
Since $\mathcal{Z}(\mathfrak{q}) \cap U \not = \emptyset$, the localizations $A_{\mathfrak{p}}$ and $A'_{\mathfrak{q}}$ are Morita equivalent by Theorem \ref{main1}.
Whence 
$$A_{\mathfrak{p}}/\mathfrak{p} \ \ \ \text{ and } \ \ \ A'_{\mathfrak{q}}/\phi( N \otimes \mathfrak{p}A_{\mathfrak{p}}M)$$ 
are Morita equivalent \cite[Theorem 5.9.ii]{McR}.
Furthermore, for $0 \leq i,j \leq n$,
\begin{equation} \label{epsiloni}
\psi_{\mathfrak{p}}(\epsilon_i \mathfrak{p}A_{\mathfrak{p}} \epsilon_j) = \epsilon'_i \mathfrak{q}A'_{\mathfrak{q}} \epsilon'_j,
\end{equation}
by Proposition \ref{homomorphism on corner}. 
Thus
$$\begin{array}{rcl}
\phi( N \otimes \mathfrak{p}A_{\mathfrak{p}} M) & = & N \cdot \zeta(\mathfrak{p}A_{\mathfrak{p}}) \cdot M \\
& = & \sum_{0 \leq i,j \leq n} A'_{\mathfrak{q}} \epsilon'_i \cdot \psi_{\mathfrak{p}}(\epsilon_i \mathfrak{p}A_{\mathfrak{p}} \epsilon_j) \cdot \epsilon'_j A'_{\mathfrak{q}} \\
& = & \sum_{0 \leq i,j \leq n} A'_{\mathfrak{q}} \epsilon'_i \cdot \epsilon'_i \mathfrak{q} A'_{\mathfrak{q}} \epsilon'_j \cdot \epsilon'_j A'_{\mathfrak{q}}\\
& = & \mathfrak{q} \sum_{0 \leq i,j \leq n} A'_{\mathfrak{q}} \epsilon'_i A'_{\mathfrak{q}} \epsilon'_j A'_{\mathfrak{q}}\\
& = & \mathfrak{q} A'_{\mathfrak{q}}.
\end{array}$$
Therefore $A_{\mathfrak{p}}/\mathfrak{p}$ and $A'_{\mathfrak{q}}/\mathfrak{q}$ are Morita equivalent.
\end{proof}

By \cite[Theorems 4.17 and 4.45]{B2}, $A$ is often not prime and nonnoetherian.
However, we have the following.

\begin{Corollary} \label{noetherian}
If $\mathcal{Z}(\mathfrak{q}) \cap U \not = \emptyset$, then the localization $A_{\mathfrak{p}}$ is prime and noetherian.
\end{Corollary}

\begin{proof}
Since $A'$ is a cancellative dimer algebra, $A'$ is prime and noetherian \cite[Theorem 3.3.3 and Proposition 3.11]{B2}.
Thus $A'_{\mathfrak{q}}$ is prime and noetherian.
But $A_{\mathfrak{p}}$ is Morita equivalent to $A'_{\mathfrak{q}}$ by Theorem \ref{main1}.
Therefore $A_{\mathfrak{p}}$ is prime and noetherian as well \cite[Proposition 5.10]{McR}.

The fact that $A_{\mathfrak{p}}$ is prime also follows directly from Lemma \ref{prime0} below and the proof of \cite[Proposition 3.11, with $A$ and $\tau$ replaced by $A_{\mathfrak{p}}$ and $\tilde{\tau}_{\mathfrak{p}}$]{B2}.
\end{proof}

\section{Azumaya loci} \label{Azumaya section}

Throughout, $A$ is a non-cancellative dimer algebra; $\psi: A \to A'$ is a cyclic contraction; and unless stated otherwise, $\mathfrak{q} \in \operatorname{Spec}S$ satisfies 
$$\mathcal{Z}(\mathfrak{q}) \cap U \not = \emptyset.$$
Set $\hat{\mathfrak{p}} := \mathfrak{q} \cap \hat{Z} \in \operatorname{Spec}\hat{Z}$.
For brevity we will denote by $\hat{Z}_{\hat{\mathfrak{p}}}/\hat{\mathfrak{p}}$ and $A_{\mathfrak{p}}/\mathfrak{p}$ the respective quotients $\hat{Z}_{\hat{\mathfrak{p}}}/\hat{\mathfrak{p}}\hat{Z}_{\hat{\mathfrak{p}}}$ and $A_{\mathfrak{p}}/\mathfrak{p}A_{\mathfrak{p}}$.

\subsection{Azumaya and smooth loci}

Recall the algebra monomorphism
$$\tau: A' \to M_{|Q'_0|}(B)$$ 
defined in (\ref{tau}).
Similarly, there is an algebra homomorphism 
\begin{equation} \label{eta}
\tilde{\tau}: A \to M_{|Q_0|}(B)
\end{equation}
defined on $p \in e_jAe_i$ by
$$p \mapsto \overbar{p} E_{ji} = \bar{\tau}\psi(p) E_{ji},$$
and extended $k$-linearly to $A$ \cite[Lemma 4.25]{B2}.
In contrast to $\tau$, $\tilde{\tau}$ is not injective \cite[Lemma 4.12]{B2}.
However, we have the following.

\begin{Lemma} \label{prime0}
There is some $\mathfrak{b} \in \operatorname{Max}B$ such that the algebra homomorphism $\tilde{\tau}$ induces an algebra monomorphism
$$\tilde{\tau}_{\mathfrak{p}}: A_{\mathfrak{p}} \to M_{|Q_0|}\left(B_{\mathfrak{b}} \right).$$
\end{Lemma}

\begin{proof}
(i) We first claim that $\tilde{\tau}_{\mathfrak{p}}$ is well-defined.
Indeed, since $A'$ is cancellative, $Z' \cong S$ by (\ref{Z' S}).
Furthermore, since $(\tau,B)$ is an impression of $A'$, the morphism $\operatorname{Max}B \to \operatorname{Max}Z' = \operatorname{Max}S$ is surjective \cite[Theorem 3.5]{B2}. 
Therefore there is an ideal $\mathfrak{b} \in \operatorname{Spec}B$ such that $\mathfrak{b} \cap S = \mathfrak{q}$ \cite[Lemma 2.15]{B}.
In particular, 
$$\mathfrak{b} \cap \hat{Z} = \hat{\mathfrak{p}}.$$
The claim then follows since $Z_{\mathfrak{p}} = \hat{Z}_{\hat{\mathfrak{p}}}$ by Lemma \ref{central iso}.

(ii) We now claim that $\tilde{\tau}_{\mathfrak{p}}$ is injective.
The $k$-linear map
$$\bar{\tau}: e_jA'e_i \to B$$ 
is injective for each $i,j \in Q_0$, by \cite[Theorem 3.5]{B2}.
Thus the kernel of $\tilde{\tau}$ is generated by elements of the form $p-q$, where $\psi(p) = \psi(q)$.

If either $\operatorname{t}(p) \not = \operatorname{t}(q)$ or $\operatorname{h}(p) \not = \operatorname{h}(q)$, then 
$$\tilde{\tau}(p) \propto E_{\operatorname{h}(p), \operatorname{t}(p)} \ \ \text{ and } \ \ \tilde{\tau}(q) \propto E_{\operatorname{h}(q), \operatorname{t}(q)}$$
have distinct non-zero matrix entries.
Whence $p-q \not \in \operatorname{ker}\tilde{\tau}$.
Thus if $p-q \in \operatorname{ker}\tilde{\tau}$, then $\operatorname{t}(p) = \operatorname{t}(q)$ and $\operatorname{h}(p) = \operatorname{h}(q)$. 
But then $p = q$ in $A_{\mathfrak{p}}$, by Theorem \ref{=0}.
Therefore $\tilde{\tau}_{\mathfrak{p}}$ is injective.
\end{proof}

\begin{Lemma} \label{same}
For each $i,j \in Q_0$,
$$\bar{\tau}\psi_{\mathfrak{p}}\left(e_iA_{\mathfrak{p}}e_i \right) = \bar{\tau}\psi_{\mathfrak{p}}\left( e_jA_{\mathfrak{p}} e_j \right).$$
\end{Lemma}

\begin{proof}
We have
$$\bar{\tau}\psi_{\mathfrak{p}}(e_iA_{\mathfrak{p}}e_i) \stackrel{\textsc{(i)}}{=} 
\bar{\tau}(e_{\psi(i)}A'_{\mathfrak{q}}e_{\psi(i)}) \stackrel{\textsc{(ii)}}{=}
\bar{\tau}(e_{\psi(j)}A'_{\mathfrak{q}}e_{\psi(j)}) \stackrel{\textsc{(iii)}}{=}
\bar{\tau}\psi_{\mathfrak{p}}(e_jA_{\mathfrak{p}}e_j).$$
Indeed, (\textsc{i}) and (\textsc{iii}) hold by Proposition \ref{homomorphism on corner}.2, and (\textsc{ii}) holds by \cite[Theorem 3.3.2]{B2}.
\end{proof}

\begin{Proposition} \label{module-finite}
$A_{\mathfrak{p}}$ is a finitely generated module over its center $Z_{\mathfrak{p}}$.
\end{Proposition}

\begin{proof}
$A_{\mathfrak{p}}$ is generated as a $Z_{\mathfrak{p}}$-module by all paths of length at most $|Q_0|$ by Lemma \ref{same} and \cite[second paragraph of proof of Theorem 2.11 with $e_iAe_i = Ze_i$ replaced by $e_iA_{\mathfrak{p}}e_i \subseteq Z_{\mathfrak{p}}e_i$]{B}.
Therefore $A_{\mathfrak{p}}$ is a finitely generated $Z_{\mathfrak{p}}$-module.
\end{proof}

By \cite[Theorem 4.50]{B2}, $A$ is nonnoetherian and an infinitely generated $Z$-module.
In contrast, we have the following.

\begin{Theorem} \label{Azumaya 1}
Suppose $\mathcal{Z}(\mathfrak{q}) \cap U \not = \emptyset$.
Then the localized algebra $A_{\mathfrak{p}}$ is prime, noetherian, and a finitely generated module over its center $Z_{\mathfrak{p}}$ with PI degree $|Q_0|$.
\end{Theorem}

\begin{proof}
$A_{\mathfrak{p}}$ is prime, noetherian, and a finitely generated module over its center by Propositions \ref{noetherian} and \ref{module-finite} respectively.
Furthermore, the algebra homomorphism $\tilde{\tau}_{\mathfrak{p}}$ is injective by Lemma \ref{prime0}.
Thus the PI degree of $A_{\mathfrak{p}}$ is $|Q_0|$ by \cite[Lemma 2.4, with $A$, $U$, $\tau_{\mathfrak{q}}$ replaced respectively by $A_{\mathfrak{p}}$, $\left\{ \mathfrak{b} \right\}$, $\tilde{\tau}_{\mathfrak{p}}$]{B}.
\end{proof}

\begin{Lemma} \label{not Azumaya lemma}
Suppose $\mathfrak{n} \not \in U$, and set $\hat{\mathfrak{m}} = \mathfrak{n} \cap \hat{Z}$.
Then $A_{\mathfrak{m}}$ is not an Azumaya algebra.
\end{Lemma}

\begin{proof}
Suppose the hypotheses hold.
Since $\mathfrak{n} \not \in U$, we have $\mathfrak{m} = \mathfrak{z}_0$ by (\ref{U z0}).

Recall that if $A_{\mathfrak{m}}$ is an Azumaya algebra, then $A_{\mathfrak{m}}/\mathfrak{m}$ is a central simple algebra over $k$ \cite[Proposition 7.11]{McR} (that is, a simple algebra whose center is $k$).
We claim that $A_{\mathfrak{z}_0}/\mathfrak{z}_0$ is not a central simple algebra.
Indeed, since $A$ is non-cancellative and $A'$ is cancellative, the contraction $\psi: A \to A'$ is non-trivial.
Thus there is at least one arrow $\delta \in Q_1$ which is contracted to a vertex.
By \cite[Lemma 4.8.1]{B2}, no cycle is contracted to a vertex.
In particular, $\delta$ is not a loop.
Whence $|Q'_0| < |Q_0|$.
Thus $|Q_0| \geq 2$.
Therefore there are at least two distinct vertex idempotents $e_1, e_2 \in A$. 

Clearly $e_1,e_2 \not \in \mathfrak{z}_0A_{\mathfrak{z}_0}$. 
Whence the (two-sided) ideal $\left\langle e_1 \right\rangle$ of $A_{\mathfrak{z}_0}/\mathfrak{z}_0$ is nonzero.
Furthermore, $\left\langle e_1 \right\rangle$ is a proper ideal since $e_2 \not \in \left\langle e_1 \right\rangle$. 
Therefore $A_{\mathfrak{z}_0}/\mathfrak{z}_0$ is not a simple algebra, and so $A_{\mathfrak{z}_0}$ is not Azumaya. 
\end{proof}

\begin{Theorem} \label{Azumaya}
The Azumaya locus $\mathcal{A} \subset \operatorname{Max}Z$ of $A$ coincides with the intersection of the Azumaya locus $\mathcal{A}' \subset \operatorname{Max}Z'$ of $A'$ and the locus $U \subset \operatorname{Max}Z'$,
$$\mathcal{A} \cong \mathcal{A}' \cap U.$$
This isomorphism is defined by sending $\mathfrak{n} \in \mathcal{A}' \cap U$ to $\mathfrak{m} \in \mathcal{A}$, where $\hat{\mathfrak{m}} = \mathfrak{n} \cap \hat{Z}$.
\end{Theorem}

\begin{proof}
Let $\mathfrak{n} \in \operatorname{Max}Z' = \operatorname{Max}S$, and set $\hat{\mathfrak{m}} := \mathfrak{n} \cap \hat{Z}$. 

If $\mathfrak{n} \not \in U$, then $A_{\mathfrak{m}}$ is not an Azumaya algebra by Lemma \ref{not Azumaya lemma}.
So suppose $\mathfrak{n} \in U$. 
Then $A_{\mathfrak{m}}/\mathfrak{m}$ and $A'_{\mathfrak{n}}/\mathfrak{n}$ are Morita equivalent by Corollary \ref{main corollary}.2.
In particular, $A_{\mathfrak{m}}/\mathfrak{m}$ is central simple over $k$ if and only if $A'_{\mathfrak{n}}/\mathfrak{n}$ is central simple over $k$.
Since $\mathfrak{n} \in U$, $A_{\mathfrak{m}}$ and $A'_{\mathfrak{n}}$ are both prime, noetherian, and finitely generated modules over their centers with PI degrees $|Q_0|$ and $|Q'_0|$ respectively, by Proposition \ref{Azumaya 1}.
Therefore $A_{\mathfrak{m}}$ is Azumaya if and only if $A'_{\mathfrak{n}}$ is Azumaya by the Artin-Processi Theorem \cite[Theorem 13.7.14]{McR}.
\end{proof}

The following corollary gives the first known class of algebras that are nonnoetherian and infinitely generated modules over their centers, with the property that their Azumaya and smooth loci coincide.
The $Y^{p,q}$ dimer algebras are defined in \cite[Example 1.3]{B}.

\begin{Corollary} \label{Azumaya corollary}
If the Azumaya and smooth loci of $A'$ coincide, then the Azumaya and smooth loci of $A$ coincide.
In particular, if $A'$ is a $Y^{p,q}$ algebra then the Azumaya and smooth loci of $A$ coincide.
\end{Corollary}

\begin{proof}
Suppose $\mathfrak{n} \not \in U$.
Then $\hat{\mathfrak{z}}_0 = \mathfrak{n} \cap \hat{Z}$ by (\ref{U z0}).
Thus $\hat{Z}_{\hat{\mathfrak{z}}_0}$ is an infinitely generated $k$-algebra \cite[Lemma 4.55]{B2}.
Whence the residue field $\hat{Z}_{\hat{\mathfrak{z}}_0}/\hat{\mathfrak{z}}_0$ has infinite projective dimension over $\hat{Z}_{\hat{\mathfrak{z}}_0}$.
Therefore $\hat{\mathfrak{z}}_0$ is a singular point of $\hat{Z}$.
The corollary then follows from Theorem \ref{Azumaya}.
If $A'$ is a $Y^{p,q}$ algebra, then its Azumaya and smooth loci coincide by \cite[Theorem 7.3]{B}.
\end{proof}

\begin{Proposition} \label{Azumaya example}
The Azumaya locus $\mathcal{A}'$ and the locus $U$ are distinct in general.
\end{Proposition}

\begin{proof}
Consider the cyclic contraction $\psi: A \to A'$ given in Figure \ref{Q}, with $Q_1^* = \left\{ \delta \right\}$.
Further, consider the cycles $p,q \in A$ drawn in red and blue respectively.
Let $V_{\rho}$ and $V_{\rho'}$ be the simple $A$- and $A'$-modules of dimensions $1^{Q_0}$ and $1^{Q'_0}$ defined by 
$$\rho(a) := \left\{ \begin{array}{cl} 1 & \text{ if } a \text{ is a subpath of } p \\ 0 & \text{ otherwise }\end{array} \right. \ \ \ \text{ for } a \in Q_0 \cup Q_1,$$ 
$$\rho'(a) := \left\{ \begin{array}{cl} 1 & \text{ if } a \text{ is a subpath of } \psi(p) \\ 0 & \text{ otherwise }\end{array} \right. \ \ \ \text{ for } a \in Q'_0 \cup Q'_1.$$
Here, we are viewing $\rho$ and $\rho'$ as vector space diagrams on $Q$ and $Q'$ respectively.
(In particular, for any path $p$ in $Q$, $\rho(p)$ is a scalar rather than a $|Q_0| \times |Q_0|$ matrix.)

Set 
$$\hat{\mathfrak{m}} := \operatorname{ann}_{\hat{Z}} V_{\rho} \in \operatorname{Max} \hat{Z} \ \ \text{ and } \ \ \mathfrak{n} := \operatorname{ann}_{Z'} V_{\rho'} \in \operatorname{Max}Z'.$$
Then $\hat{\mathfrak{m}} = \mathfrak{n} \cap \hat{Z}$ under the isomorphism $Z' \cong S$.
We claim that $\mathfrak{n}$ is in $U$, but not in the Azumaya locus of $A'$.

To show that $\mathfrak{n} \in U$, it suffices to show that if $\overbar{s}$ is a monomial in $S \setminus \hat{Z}$, then $\overbar{s}$ is also in $\hat{Z}_{\hat{\mathfrak{m}}}$. 
Since $\sigma \in \hat{Z}$, we may suppose $s$ is a cycle in $\mathcal{C}^u$ with $u \in \mathbb{Z}^2 \setminus 0$.

First note that the only cycle in $\mathcal{C}^u$ which does not share a vertex subpath with $p$ is the cycle $q$ (drawn in blue).
But $\overbar{q} = \overbar{p}$ by \cite[Theorem 3.7]{B}.
So consider a cycle $s \in \mathcal{C}^u$ which shares a vertex subpath with $p$, say at vertex $i \in Q_0$.
(For example, we may take $s$ to be the green cycle in the figure.)
Denote by $p_i$ and $s_i$ the cyclic permutations of $p$ and $s$ with tails at $i$.
Then by the symmetry of $Q$, it is clear that
$$\overbar{s} \overbar{p} = \overbar{s}_i\overbar{p}_i = \overbar{s_ip_i} \in R.$$

Furthermore, it is straightforward to verify that there are no central elements in the kernel of $\psi$ in this example.
In particular, $\operatorname{nil}Z = 0$ by \cite[Theorem 4.24]{B2}.
Thus $\hat{Z} \cong R$.
Consequently, $\overbar{s} \overbar{p} \in \hat{Z}$. 
Since $p$ does not annihilate $V_{\rho}$, the monomial $\overbar{p}$ is in $\hat{Z} \setminus \hat{\mathfrak{m}}$.
Thus 
$$\overbar{s} = \overbar{s}\overbar{p} \overbar{p}^{-1} \in \hat{Z}_{\hat{\mathfrak{m}}}.$$
Therefore $\hat{Z}_{\hat{\mathfrak{m}}} = S_{\mathfrak{n}}$.
Whence $\mathfrak{n} \in U$.

Finally, $\mathfrak{n}$ is not in the Azumaya locus of $A'$ since the dimension vector of any simple $A'$-module of maximal $k$-dimension is $1^{Q'_0}$ \cite[Proposition 2.5, Lemma 2.13]{B}.
\end{proof}

\begin{figure}
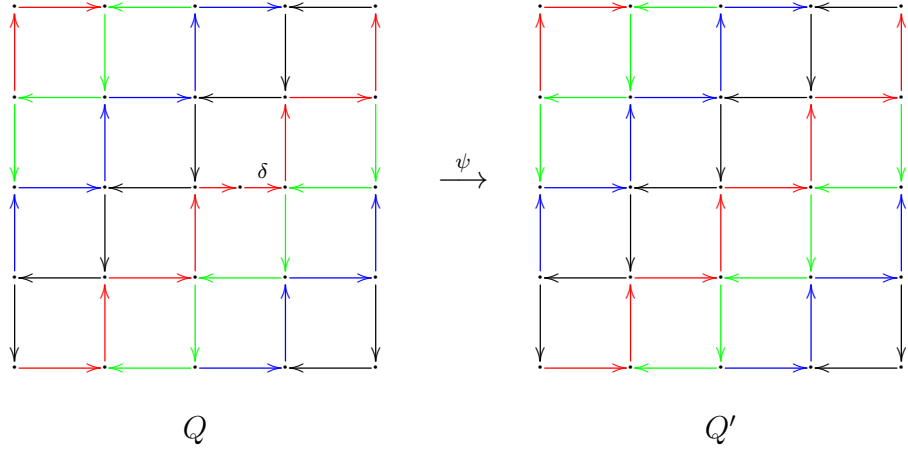

$$\begin{array}{ccc}
\xy
(-24,-24)*{\cdot}="1";(-12,-24)*{\cdot}="2";(0,-24)*{\cdot}="3";(12,-24)*{\cdot}="4";(24,-24)*{\cdot}="5";
(-24,-12)*{\cdot}="6";(-12,-12)*{\cdot}="7";(0,-12)*{\cdot}="8";(12,-12)*{\cdot}="9";(24,-12)*{\cdot}="10";
(-24,0)*{\cdot}="11";(-12,0)*{\cdot}="12";(0,0)*{\cdot}="13";(12,0)*{\cdot}="14";(24,0)*{\cdot}="15";
(-24,12)*{\cdot}="16";(-12,12)*{\cdot}="17";(0,12)*{\cdot}="18";(12,12)*{\cdot}="19";(24,12)*{\cdot}="20";
(-24,24)*{\cdot}="21";(-12,24)*{\cdot}="22";(0,24)*{\cdot}="23";(12,24)*{\cdot}="24";(24,24)*{\cdot}="25";
(6,0)*{\cdot}="26";
{\ar@[red]"13";"26"};{\ar^{\delta}@[red]"26";"14"};
{\ar@[red]"1";"2"};{\ar@[green]"3";"2"};{\ar@[blue]"3";"4"};{\ar@{->}"5";"4"};
{\ar@{->}"7";"6"};{\ar@[red]"7";"8"};{\ar@[green]"9";"8"};{\ar@[blue]"9";"10"};
{\ar@[blue]"11";"12"};{\ar@{->}"13";"12"};{\ar@[green]"15";"14"};
{\ar@[green]"17";"16"};{\ar@[blue]"17";"18"};{\ar@{->}"19";"18"};{\ar@[red]"19";"20"};
{\ar@[red]"21";"22"};{\ar@[green]"23";"22"};{\ar@[blue]"23";"24"};{\ar@{->}"25";"24"};
{\ar@{->}"6";"1"};{\ar@[blue]"6";"11"};{\ar@[green]"16";"11"};{\ar@[red]"16";"21"};
{\ar@[red]"2";"7"};{\ar@{->}"12";"7"};{\ar@[blue]"12";"17"};{\ar@[green]"22";"17"};
{\ar@[green]"8";"3"};{\ar@[red]"8";"13"};{\ar@{->}"18";"13"};{\ar@[blue]"18";"23"};
{\ar@[blue]"4";"9"};{\ar@[green]"14";"9"};{\ar@[red]"14";"19"};{\ar@{->}"24";"19"};
{\ar@{->}"10";"5"};{\ar@[blue]"10";"15"};{\ar@[green]"20";"15"};{\ar@[red]"20";"25"};
\endxy
& \ \ \ \stackrel{\psi}{\longrightarrow} \ \ \ &
\xy
(-24,-24)*{\cdot}="1";(-12,-24)*{\cdot}="2";(0,-24)*{\cdot}="3";(12,-24)*{\cdot}="4";(24,-24)*{\cdot}="5";
(-24,-12)*{\cdot}="6";(-12,-12)*{\cdot}="7";(0,-12)*{\cdot}="8";(12,-12)*{\cdot}="9";(24,-12)*{\cdot}="10";
(-24,0)*{\cdot}="11";(-12,0)*{\cdot}="12";(0,0)*{\cdot}="13";(12,0)*{\cdot}="14";(24,0)*{\cdot}="15";
(-24,12)*{\cdot}="16";(-12,12)*{\cdot}="17";(0,12)*{\cdot}="18";(12,12)*{\cdot}="19";(24,12)*{\cdot}="20";
(-24,24)*{\cdot}="21";(-12,24)*{\cdot}="22";(0,24)*{\cdot}="23";(12,24)*{\cdot}="24";(24,24)*{\cdot}="25";
{\ar@[red]"1";"2"};{\ar@[green]"3";"2"};{\ar@[blue]"3";"4"};{\ar@{->}"5";"4"};
{\ar@{->}"7";"6"};{\ar@[red]"7";"8"};{\ar@[green]"9";"8"};{\ar@[blue]"9";"10"};
{\ar@[blue]"11";"12"};{\ar@{->}"13";"12"};{\ar@[red]"13";"14"};{\ar@[green]"15";"14"};
{\ar@[green]"17";"16"};{\ar@[blue]"17";"18"};{\ar@{->}"19";"18"};{\ar@[red]"19";"20"};
{\ar@[red]"21";"22"};{\ar@[green]"23";"22"};{\ar@[blue]"23";"24"};{\ar@{->}"25";"24"};
{\ar@{->}"6";"1"};{\ar@[blue]"6";"11"};{\ar@[green]"16";"11"};{\ar@[red]"16";"21"};
{\ar@[red]"2";"7"};{\ar@{->}"12";"7"};{\ar@[blue]"12";"17"};{\ar@[green]"22";"17"};
{\ar@[green]"8";"3"};{\ar@[red]"8";"13"};{\ar@{->}"18";"13"};{\ar@[blue]"18";"23"};
{\ar@[blue]"4";"9"};{\ar@[green]"14";"9"};{\ar@[red]"14";"19"};{\ar@{->}"24";"19"};
{\ar@{->}"10";"5"};{\ar@[blue]"10";"15"};{\ar@[green]"20";"15"};{\ar@[red]"20";"25"};
\endxy
\\ \\
Q & \ \ \ \ \ \ & Q'
\end{array}$$
\caption{The quivers $Q$ and $Q'$ in Proposition \ref{Azumaya example} and Remark \ref{first remark}, drawn on a torus.  
The cycles $p$, $q$, $s$ in $Q$ are drawn in red, blue, and green respectively, with suitably chosen tails.}
\label{Q}
\end{figure}

\subsection{Classification of simple modules parameterized by the Azumaya locus}

Given a quiver algebra $A = kQ/I$ and dimension vector $d = (d_i)_{i \in Q_0}$, denote by $\operatorname{Rep}_d(A)$ the closed affine variety of $d$-dimensional representations of $A$ viewed as vector space diagrams on $Q$,
$$\operatorname{Rep}_d(A) \subset \bigoplus_{a \in Q_1} M_{d_{\operatorname{h}(a)} \times d_{\operatorname{t}(a)}}\left(k \right) = \mathbb{A}_{k}^{\sum_{a \in Q_1} d_{\operatorname{h}(a)}d_{\operatorname{t}(a)}}.$$

Now let $\psi: A \to A'$ be a cyclic contraction of dimer algebras.
For the following, consider simple representations $\rho \in \operatorname{Rep}_{1^{Q_0}}(A)$ and $\rho' \in \operatorname{Rep}_{1^{Q'_0}}(A')$.
Recall that $\rho(\delta) \not = 0$ for each $\delta \in Q_1^*$, by Proposition \ref{delta*}.

Set $\rho_0 := \rho$.
For each $n \geq 1$, define the representation $\rho_n \in \operatorname{Rep}_{1^{Q_0}}(A)$ iteratively on $a \in Q_1$ by 
$$\rho_{n+1}(a) := \rho_n(a) \prod_{\substack{\delta \in Q_1^* \\ \operatorname{h}(\delta) = \operatorname{t}(a)}} \rho_n(\delta) \prod_{\substack{\delta' \in Q_1^* \\ \operatorname{h}(\delta') = \operatorname{h}(a)}} \rho_n(\delta')^{-1}.$$
Since no unoriented cycle is contracted to a vertex by \cite[Lemma 4.8.1]{B2}, there is an $N \geq 1$ such that for each $n \geq N$, 
$$\rho_n = \rho_N \ \ \ \text{ and } \ \ \ \rho_N(\delta) = 1 \ \text{ for each } \ \delta \in Q_1^*.$$
Set $\rho^* := \rho_N$.
Clearly $\rho^*$ and $\rho$ are isomorphic representations of $A$.

Consider the map
\begin{equation} \label{psi-1}
\psi^{-1}: kQ' \to \epsilon_0 kQ \epsilon_0
\end{equation}
defined by $\psi^{-1}(a) = b$ where $\psi(b) = a$.
$\psi^{-1}$ is a well-defined algebra homomorphism by Proposition \ref{homomorphism on corner}.1.
We may thus define representations $\psi^{-1}\rho' \in \operatorname{Rep}_{1^{Q_0}}(A)$ and $\psi \rho \in \operatorname{Rep}_{1^{Q'_0}}(A')$ by 
$$( \psi^{-1}\rho')(a) := \rho'(\psi(a)) \ \ \text{ for each } a \in Q_1$$
and
$$\left( \psi \rho \right)(a) := \rho^*( \psi^{-1}(a)) \ \ \text{ for each } a \in Q'_1.$$
Note that $(\psi^{-1}\rho')^* = \psi^{-1}\rho'$.

\begin{Lemma} \label{psi rho}
$\rho^* = \psi^{-1} \rho'$ if and only if $\psi \rho = \rho'$.
\end{Lemma}

\begin{proof}
First suppose $\rho^* = \psi^{-1}\rho'$, and let $a \in Q_1$.
Then
$$(\psi\rho)(a) = \rho^*( \psi^{-1}(a)) = (\psi^{-1}\rho')(\psi^{-1}(a)) = \rho'(\psi(\psi^{-1}(a))) = \rho'(a).$$
Conversely suppose $\psi \rho = \rho'$, and let $a \in Q'_1$.
Then
$$(\psi^{-1} \rho')(a) = \rho'(\psi(a)) = (\psi \rho)(\psi(a)) = \rho^*(\psi^{-1}(\psi(a))) = \rho^*(a).$$
\end{proof}

Recall that a simple $A$-module $V$ is said to sit over a point $\mathfrak{m}$ in the Azumaya locus $\mathcal{A}$ if $V_{\mathfrak{m}} := A_{\mathfrak{m}} \otimes_A V$ is the unique simple $A_{\mathfrak{m}}$-module up to isomorphism.
The Azumaya locus then parameterizes a family of simple $A$-module isoclasses.
  
\begin{Proposition} \label{simple 1}
$V = V_{\rho}$ is a simple $A$-module of dimension $1^{Q_0}$ if and only if $V$ sits over some point $\mathfrak{m} \in \mathcal{A}$.
\end{Proposition}

\begin{proof}
($\Leftarrow$) 
First suppose $\mathfrak{m} \in \mathcal{A}$. 
Let $V_{\mathfrak{m}}$ be the unique simple $A_{\mathfrak{m}}$-module.
We claim that $V$ has dimension $1^{Q_0}$.

By Theorem \ref{Azumaya 1}, the PI degree of $A_{\mathfrak{m}}$ is $|Q_0|$.
Thus 
\begin{equation} \label{Vm}
\operatorname{dim}_k \left( V_{\mathfrak{m}} \right) = |Q_0|.
\end{equation}

Let $\mathfrak{n} \in \operatorname{Max}S$ be such that $\mathfrak{n} \cap Z = \mathfrak{m}$.
Since $\mathfrak{m} \in \mathcal{A}$, we have $\mathfrak{n} \in \mathcal{A}' \cap U$ by Theorem \ref{Azumaya}.
But $\mathfrak{n} \in U$ implies $\mathfrak{m} \not = \mathfrak{z}_0$ by (\ref{U z0}). 
Thus there is a monomial $z \in Z \setminus \mathfrak{m}$.
In particular, $\rho(z) \not = 0$.

Since $V$ is simple and $z$ is central, $\rho(z)$ is a scalar multiple of the identity by Schur's lemma.
Whence $\rho(ze_i) \not = 0$ for each $i \in Q_0$.
Thus $\operatorname{dim}_k \left( e_i V_{\mathfrak{m}} \right) \geq 1$ for each $i \in Q_0$.
Therefore $V_{\mathfrak{m}}$ has dimension $1^{Q_0}$ by (\ref{Vm}).

($\Rightarrow$) 
Now suppose $V_{\rho}$ is simple of dimension $1^{Q_0}$.
We claim that $V$ sits over a point in $\mathcal{A}$.
 
Consider $\rho':= \psi \rho$ as in Lemma \ref{psi rho}.
Set $\mathfrak{m} := \operatorname{ann}_Z V_{\rho}$ and $\mathfrak{n} := \operatorname{ann}_{Z'} V_{\rho'}$.
Recall that $V_{\rho'}$ is simple of dimension $1^{Q'_0}$.
In particular,
\begin{equation} \label{1}
\mathfrak{n} \in \mathcal{A}'.
\end{equation}

Furthermore, since $V_{\rho}$ is simple of dimension $1^{Q_0}$, $\mathfrak{m} \not = \mathfrak{z}_0$.
Thus by (\ref{U z0}), 
\begin{equation} \label{2}
\mathfrak{n} \in U.
\end{equation}
It then follows from (\ref{1}), (\ref{2}), and Theorem \ref{Azumaya} that $\mathfrak{m} \in \mathcal{A}$.
\end{proof}

In the following, the algebra homomorphism $\tilde{\tau}$ defined in (\ref{eta}) is used to classify the simple $A$-modules parameterized by the Azumaya locus. 
This classification shows that $\tilde{\tau}$ is very close to being an impression of $A$ even though $A$ may not embed into a matrix ring over a commutative ring; see \cite[Proposition 2.5]{B}.

\begin{Theorem} \label{almost impression}
For each $A$-module $V_{\rho}$ that sits over a point in the Azumaya locus $\mathcal{A}$, there is a point $\mathfrak{b} \in \operatorname{Max}B$ such that $\rho$ is isomorphic to the composition
$$A \stackrel{\tilde{\tau}}{\to} M_{|Q_0|}(B) \stackrel{\epsilon_{\mathfrak{b}}}{\to} M_{|Q_0|}\left(B/\mathfrak{b} \right).$$
\end{Theorem}

\begin{proof}
Suppose $V_{\rho}$ is an $A$-module which sits over a point in $\mathcal{A}$.
Then by Proposition \ref{simple 1}, $V_{\rho}$ is a simple $A$-module of dimension $1^{Q_0}$.
By Lemma \ref{psi rho}, there is a simple $A'$-module $V_{\rho'}$ of dimension $1^{Q'_0}$ such that $\psi^{-1} \rho' = \rho^* \cong \rho$.
Since $(\tau,B)$ is an impression of $A$, by \cite[Proposition 2.5]{B} there is a point $\mathfrak{b} \in \operatorname{Max}B$ such that $\rho'$ is isomorphic to the composition  
$$A \stackrel{\tau}{\hookrightarrow} M_{|Q'_0|}(B) \stackrel{\epsilon_{\mathfrak{b}}}{\to} M_{|Q'_0|}\left(B/\mathfrak{b} \right).$$
But then for each $i,j \in Q_0$,
$$\rho|_{e_jAe_i} \cong \left( \psi^{-1} \rho' \right)|_{e_jAe_i} = \rho' \psi|_{e_jAe_i} \cong \epsilon_{\mathfrak{b}} \tau \psi|_{e_jAe_i} = \epsilon_{\mathfrak{b}} \tilde{\tau}|_{e_jAe_i}.$$
\end{proof}

\subsection{The cycle algebra is unique} \label{coordinate-free}

\begin{Definition}\label{S(A)} \rm{
Let $A$ be a dimer algebra.
Denote by $\mathbb{S}(A)$ the open subvariety of $\operatorname{Rep}_{1^{Q_0}}(A)$ of simple representations, and by $\overbar{\mathbb{S}(A)}$ its Zariski closure.
For an element $p$ in a corner ring $e_jAe_i$, denote by $\mu(p)$ the corresponding function in $k\left[ \operatorname{Rep}_{1^{Q_0}}(A) \right]$ taking the value 
$$\mu(p)(\rho) := \rho(p) \in k$$ 
on each $\rho \in \operatorname{Rep}_{1^{Q_0}}(A)$.
} \end{Definition}

\begin{Remark} \label{first remark} \rm{
It possible for the closure $\overbar{\mathbb{S}(A)}$ of $\mathbb{S}(A)$ to be properly contained in $\operatorname{Rep}_{1^{Q_0}}(A)$.
Indeed, $\overbar{\mathbb{S}(A)} \not = \operatorname{Rep}_{1^{Q_0}}(A)$ if there are cycles $p,q \in A$ and a representation $\rho \in \operatorname{Rep}_{1^{Q_0}}(A)$ such that 
$$\overbar{p} = \overbar{q} \ \ \ \text{ and } \ \ \ \rho(p) \not = \rho(q) \ \ \text{(as scalars)},$$ 
by Theorem \ref{almost impression}.

For example, let $A$ be the dimer algebra with quiver $Q$ given in Figure \ref{Q}.
The paths $p$ and $q$, drawn in red and blue respectively, satisfy $\overbar{p} = \overbar{q}$ by \cite[Example 3.9]{B2}.
However, consider the semisimple representation $\rho \in \operatorname{Rep}_{1^{Q_0}}(A)$ where each arrow subpath of $p$ is represented by $1$, each arrow subpath of $q$ is represented by $2$, and all other arrows are represented by zero.
Then $\rho(p) = 1 \not = 2^8 = \rho(q)$.
}\end{Remark}

Recall that the reductive algebraic group 
$$\operatorname{GL} := \prod_{j \in Q_0} \operatorname{GL}_{d_j}(k)$$
acts linearly on $\operatorname{Rep}_d(A)$ by conjugation.

\begin{Theorem} \label{cycle algebra theorem}
Suppose $\psi: A \to A'$ is a cyclic contraction.
Then the cycle algebra $S$ is isomorphic to the $\operatorname{GL}$-invariant rings
\begin{equation} \label{S=}
S = k[ \overbar{\mathbb{S}(A)} ]^{\operatorname{GL}} = k[ \overbar{\mathbb{S}(A')} ]^{\operatorname{GL}}.
\end{equation}
\end{Theorem}

\begin{proof}
Each arrow $a \in Q_1$ vanishes at some point of $\overbar{\mathbb{S}(A)}$, and so $\mu(a)$ is not invertible on $\overbar{\mathbb{S}(A)}$ (though if $a \in Q_1^{\mathcal{S}}$, then $\mu(a)$ is invertible on $\mathbb{S}(A)$).
Therefore the $\operatorname{GL}$-invariants in $k[ \overbar{\mathbb{S}(A)} ]$ and $k[ \overbar{\mathbb{S}(A')}]$ are generated by the $\mu$-images of oriented cycles in $Q$ and $Q'$ respectively.

By Proposition \ref{simple 1} and Theorem \ref{almost impression}, for each cycle $p \in A$ we may set 
$$\mu(p) = \overbar{p}.$$ 
Therefore (\ref{S=}) holds.
\end{proof}

\begin{Remark} \rm{
The `mesonic chiral ring' in an abelian quiver gauge theory is the ring of gauge invariant operators defined on the vacuum moduli space.
Morally, the mesonic chiral ring is then the ring of invariants $k[\operatorname{Rep}_{1^{Q_0}}(A)]^{\operatorname{GL}}$. 
However, the mesonic chiral ring may not coincide with the cycle algebra $S$ by Remark \ref{first remark}.
In the example therein, $\mu(p) \not = \mu(q)$ in $k[\operatorname{Rep}_{1^{Q_0}}(A)]^{\operatorname{GL}}$, whereas $\mu(p) = \mu(q)$ in $S$.
}\end{Remark}

We conclude with a curious example of a dimer algebra $A$ for which $\mathbb{S}(A)$ consists of only two points.

\begin{Non-example} \label{non-ex} \rm{
Consider the dimer algebra $A$ with quiver $Q$ given in Figure \ref{non-example dimer}.
$Q$ consists of one vertex and four loops $a,b,c,d$, and has no perfect matchings.
$A$ has two permanent 2-cycles, and so does not admit a contraction (cyclic or not) to a cancellative dimer algebra \cite[Proposition 4.43]{B2}.
Furthermore, $A$ has only two representations $\rho_0$ and $\rho_1$ of dimension vector $1^{Q_0} = 1$, namely
$$\rho_0(a) = \rho_0(b) = \rho_0(c) = \rho_0(d) = 0$$
and
$$\rho_1(a) = \rho_1(b) = \rho_1(c) = \rho_1(d) = 1.$$
In particular, 
$$\operatorname{Rep}_{1}(A) = \mathbb{S}(A) = \left\{ \rho_0, \rho_1 \right\}.$$
}\end{Non-example}

\begin{figure}
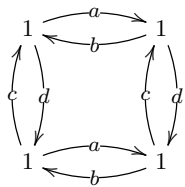

$$\xy 0;/r.3pc/:
(-7,7)*+{\text{\scriptsize{$1$}}}="1";(7,7)*+{\text{\scriptsize{$1$}}}="2";
(-7,-7)*+{\text{\scriptsize{$1$}}}="4";(7,-7)*+{\text{\scriptsize{$1$}}}="3";
{\ar@/^/|-a"1";"2"};{\ar@/^/|-d"2";"3"};{\ar@/^/|-b"3";"4"};{\ar@/^/|-c"4";"1"};{\ar@/^/|-b"2";"1"};{\ar@/^/|-d"1";"4"};{\ar@/^/|-a"4";"3"};{\ar@/^/|-c"3";"2"};
\endxy$$
\caption{The quiver $Q$ in Non-example \ref{non-ex}, drawn on a torus.  The representation space $\operatorname{Rep}_1(A) = \mathbb{S}(A)$ of $A = kQ/I$ consists of only two points.}
\label{non-example dimer}
\end{figure}

\ \\
\textbf{Acknowledgments.}  Part of this article is based on work supported by the Simons Foundation while the author was a postdoc at the Simons Center for Geometry and Physics at Stony Brook University. 

\bibliographystyle{hep}
\def\cprime{$'$} \def\cprime{$'$}

\end{document}